\documentclass[11pt,psamsfonts]{amsart} 
\usepackage{amsmath}
\usepackage{amsthm}
\usepackage{amssymb}
\usepackage{amscd}
\usepackage{amsfonts}
\usepackage{amsbsy}
\usepackage{mathrsfs} 
\usepackage{setspace}
\usepackage{graphicx}
\usepackage{color}
\usepackage{calc}
\usepackage{subfigure}
\usepackage{caption}
\usepackage[font=small,labelfont=bf]{caption}
\usepackage{epsfig,afterpage}
\usepackage[dvips]{psfrag}
\usepackage[numbers,sort&compress]{natbib}

\advance\textwidth by +1.5in \advance\textheight by +1.4in
\advance\topmargin by -0.6in \advance\oddsidemargin by -0.8in
\advance\evensidemargin by -0.8in

\newcommand{\N}{\ensuremath{\mathbb{N}}}

\newcommand{\II}{\ensuremath{\mathbb{I}}}

\newtheorem {theorem} {Theorem}[section]
\newtheorem {proposition} [theorem]{Proposition}

\newtheorem {lemma} [theorem]{Lemma}
\newtheorem {example} [theorem]{Example}

\newtheorem {definition} [theorem]{Definition}
\begin{document}

\title[Limit cycles in refracted Hamiltonian systems with a straight switching line]
{Limit cycles in refracted Hamiltonian systems with a straight switching line}

\author[A. Bakhshalizadeh and A. C. Rezende]
{Ali Bakhshalizadeh$^1$ and Alex C. Rezende$^2$}

\address{$^{1, 2}$ Departamento de Matem\'atica, Universidade Federal de S\~{a}o Carlos, S\~{a}o Carlos, Brazil.} 
\email{alibb@ufscar.br}
\email{alexcr@ufscar.br}



\maketitle

\begin{abstract}
This paper presents a criterion that provides an easy sufficient condition for a collection of line integrals to have the Chebyshev property. 
The condition is based on the functions appearing in the line integrals. The criterion is used to study the number of limit cycles in 
refracted differential systems, which are formed by two Hamiltonian differential systems separated by a straight line. The paper concludes by
presenting new results on such systems, which show the effectiveness of the criterion presented.
\end{abstract}

\section{Introduction and main results}\label{sec1}
Piecewise differential systems are mathematical models that are used to represent systems with non-smooth or discontinuous behavior. 
These systems are particularly useful in fields such as physics, engineering, epidemiology, and economics where discontinuity is a common 
phenomenon \cite{BBCK,CLBB,Kunze,ML,TOBSA,TLXC}. They refer to systems of differential equations that are defined in different zones of
the state space and are separated by boundaries. These boundaries, also known as switchings or jump conditions, can cause 
non-smoothness or discontinuity in the system, making it challenging to analyze and control. The solution on the switching boundries is 
defined by A. F. Filippov \cite{Fil}. It is a mathematical framework for describing the behavior of systems that exhibit discontinuities 
or switches among different zones.

The study of limit cycles in piecewise Hamiltonian differential systems has gained a lot of attention in the recent years. 
Indeed the study of limit cycles in such systems is an extension of the famous Hilbert's 16th problem which asks about the number of limit 
cycles in the planar polynomial differential systems of degree $n+1$. 
There are many papers dealing with the number of limit cycles of the piecewise
polynomial Hamiltonian systems with a straight switching line (see for instance \cite{YZ,LL,CLYZ}). The Melnikov theory and averaging
theory are widely useful methods for investigating the number of limit cycles in piecewise smooth systems, and they are frequently applied 
in many articles. The techniques mentioned were first established in \cite{LH} and \cite{LMN}, respectively.

The Chebyshev property of the family of functions plays a crucial role in the study of limit cycles in differential 
systems, 
as it enables us to determine the number of zeros of the first order Melnikov function. 
Specifically, by examining the number of real zeros in any nontrivial linear combination of these functions, which 
constitutes the first order Melnikov
function, we can obtain an exact upper bound on the number of limit cycles in the system. Therefore, the 
Chebyshev property provides a valuable 
tool for studying the number of limit cycles in differential systems, see \cite{Pe,KS,Ma,GI,GWLZ,GMV,MV}. 
Here our primary focus is 
on investigating the number of limit cycles in refracted Hamiltonian differential systems consisting of two zones 
separated by a straight line 
at $x=0$. We introduce a criterion that presents an easy sufficient condition for a family of line integrals to exhibit the 
Chebyshev property. It is important to note that this criterion is not universally applicable; however, when it is applicable, it can 
significantly simplify the solution process.

In this paper we consider a classical Hamiltonian function given by
\begin{equation}\label{Hm}
H(x,y):=\chi(x)y^{2}+\Psi(x):=\begin{cases}
H^{+}(x,y)=\chi_{1}(x)y^{2}+\Psi_{1}(x), & x> 0,\\
H^{-}(x,y)=\chi_{2}(x)y^{2}+\Psi_{2}(x), & x< 0, 
\end{cases}
\end{equation}
where $ H^{\pm}(x,y) $ are analytic function in some open subset of the plane, and $ \Psi_{1} (0)=\Psi_{2} (0)=0$. 
We also assume that the following inequalities hold:
\begin{equation*}\label{H1}
\begin{array}{c}
x\Psi_{1}^{\prime}(x), \chi_{1}(x)>0,\qquad \text{for all}\qquad x\in(0,x_{r}), \\
x\Psi_{2}^{\prime}(x), \chi_{2}(x)>0,\qquad \text{for all}\qquad x\in(x_{l},0), \tag{H1}
\end{array}
\end{equation*}
and $ \lim_{x\rightarrow 0^+}\chi_{1}(x)=\lim_{x\rightarrow 0^-}\chi_{2}(x)>0 $.
With these assumptions it is easy to see that $(0, 0)$ is a local minimum,
and there exists a punctured neighborhood of the origin that is foliated by periodic orbits.
Using assumption \eqref{H1}, we can deduce the existence of two analytic functions $\sigma_1$ and $\sigma_2$ 
satisfying
\begin{equation*}
\begin{aligned}
\Psi_{1}(x)=\Psi_{2}(\sigma_{1}(x)),\qquad\text{for all}\qquad x\in(0, x_{r}),\\
\Psi_{2}(x)=\Psi_{1}(\sigma_{2}(x)),\qquad\text{for all}\qquad x\in(x_{l}, 0),
\end{aligned}
\end{equation*}
where $\sigma_1$ and $\sigma_2$ are non-identity analytic functions satisfying $\sigma_1 \circ \sigma_2 = 
\text{Id}$ on $(x_l,0)$ and $\sigma_2 \circ \sigma_1 = \text{Id}$ on $(0,x_r)$. Now, we define an involution 
$\sigma(x)$ on $(x_l,x_r)$ as
\begin{equation*}
\sigma(x) = 
\begin{cases}
\sigma_2(x), & x\in(x_l,0),\\ 
0, & x=0,\\
\sigma_1(x), & x\in(0,x_r),
\end{cases} 
\end{equation*}
which is clear that under the above assumption $\sigma(x)$ is an involution on $(x_l,x_r)$. Note that a mapping
$\sigma:\II\rightarrow \II$ is an {\it involution} if $\sigma \circ \sigma = \text{Id}$ and $\sigma \neq \text{Id}$. 
Let $\kappa$ be a function defined on the interval $\II\setminus \lbrace 0\rbrace$ as
\begin{equation*} 
\kappa(x) = 
\begin{cases}
\kappa_{1}(x), & x\in \II^{+}, \\ 
\kappa_{2}(x), & x\in \II^{-},
\end{cases} 
\end{equation*} 
where $\II^+$ and $\II^-$ are the positive and negative intervals of the real line, respectively. 
Then we define its {\it balance} with respect to the involution $\sigma$ as follows
\begin{equation*}
\mathscr{B}_{\sigma}(\kappa)(x)=\dfrac{\kappa(x)-\kappa(\sigma(x))}{2},
\end{equation*}
or equivalently
\begin{equation*}
\mathscr{B}_{\sigma}(\kappa)(x)=\begin{cases}
\dfrac{\kappa_{1}(x)-\kappa_{2}(\sigma_{1}(x))}{2}, & x\in \II^{+}, \\ 
\dfrac{\kappa_{2}(x)-\kappa_{1}(\sigma_{2}(x))}{2}, & x\in \II^{-}.
\end{cases} 
\end{equation*}
The balance of a function $\kappa$ with respect to an involution $\sigma$ is related to the odd part of $\kappa$. 
If $\kappa$ is a balanced function, then its odd part is identically zero. Conversely, if the odd part of $\kappa$ is identically 
zero, then $\kappa$ is a balanced function. 

The main results of the paper can be given as follows. In the following theorems, we consider line integrals where the piecewise Hamiltonian differential systems are 
consisted of two distinct zones.
\begin{theorem}\label{the1}
Consider the following line integrals
\begin{equation*} 
L_i (h)=
\begin{cases}
I_\frac{i}{2}(h), & i = 0, 2, \ldots, 2n-2, \\ 
J_\frac{i-1}{2}(h), & i = 1, 3, \ldots, 2n-1,
\end{cases} 
\end{equation*} 
with
\begin{equation*}
I_{i}(h)=\int_{\Gamma_{h}^{+}}f_{1i}(x)y^{2s-1} dx,\qquad
J_{i}(h)=\int_{\Gamma_{h}^{-}}f_{2i}(x)y^{2s-1} dx,
\end{equation*}
where, for each $h\in(0,h_{0})$, $\Gamma_{h}^{+}\cup\Gamma_{h}^{-}$ is the oval around the origin defined 
as 
\begin{equation*}
\Gamma^{+}_{h}\subset\lbrace \chi_{1}(x)y^{2}+\Psi_{1}(x)=h\mid x>0\rbrace,\qquad
\Gamma^{-}_{h}\subset\lbrace \chi_{2}(x)y^{2}+\Psi_{2}(x)=h\mid x<0\rbrace.
\end{equation*}
Let $\sigma$ be the involution associated with $\Psi_{1}$ and $ \Psi_{2} $, and define
\begin{equation*} 
l_i(x)= 
\begin{cases}
\left( \dfrac{f_{1\left( \frac{i}{2}\right)} }{\chi_{1}^{\frac{2s-1}{2}}}\right) (x), & i=0,2,\ldots,2n-2, \\ 
\\
-\left( \dfrac{f_{2\left( \frac{i-1}{2}\right) }}
{\chi_{2}^{\frac{2s-1}{2}}}\right)(\sigma(x))\sigma^{\prime}(x), & i=1,3,\ldots,2n-1.
\end{cases} 
\end{equation*} 
Then $ \lbrace L_{0},L_{1},\ldots,L_{2n-1}\rbrace $ is an ECT-system on $(0,h_{0})$ if 
$ \lbrace l_{0},l_{1},\ldots,l_{2n-1}\rbrace$ is a CT-system on $(0,x_{r})$ and $s>2(n-1)$.
\end{theorem}

\begin{theorem}\label{the2}
Consider the following line integrals
\begin{equation*}
I_{i}(h)=\int_{\Gamma_{h}^{+}}f_{i}(x)y^{2s-1} dx,\qquad i=0,1,\ldots,n-1,
\end{equation*}
where, for each $h\in(0,h_{0})$, $\Gamma_{h}^{+}$ is the arc with the level curve $ \lbrace 
\chi(x)y^{2}+\Psi_{1}(x)=h, x>0\rbrace $. Let $\sigma$ be the involution associated with $\Psi_{1}$ and $ \Psi_{2} $, and $\sigma^{\prime}(x)$ is a constant function 
on $(x_{l},x_r)\setminus\lbrace 0 \rbrace$, and define
\begin{equation*}
l_{i}(x)= \left( \dfrac{f_{i}}{\chi_{1}^{\frac{2s-1}{2}}}\right) (x).
\end{equation*}
Then $ \lbrace I_{0},I_{1},\ldots,I_{n-1}\rbrace $ is an ECT-system on $(0,h_{0})$ if 
$\lbrace l_{0}, l_{1},\dots, l_{n-1}\rbrace$ is a CT-system on $(0,x_{r})$ and $s>n-2$.
\end{theorem}

Note that, by applying Lemma \ref{lem4}, we can show that the set of functions $\lbrace f_0, f_1,\dots, f_{n-1}\rbrace$ is an ECT-system on 
$(0,x_{r})$, which also implies that $\lbrace f_0, f_1,\dots, f_{n-1}\rbrace$ is also a CT-system on $(0,x_{r})$. Also, if the 
assumptions $s>2(n-1)$ and $s>n-2$ are not satisfied in Theorems \ref{the1} and \ref{the2}, respectively, we can use Lemma \ref{lems} 
to promote the value of $s$. 

The paper is structured as follows. In the second section, we introduce essential tools and notations that will be used to prove our main
results. First, we explain several concepts and definitions related to piecewise smooth systems,
and we introduce the first-order Melnikov function for piecewise smooth systems. Also, we present the Chebyshev
property, which is a crucial concept in the study of limit cycles of dynamical systems. In the third section, we apply the tools and notations introduced in the previous 
section to prove the main results of our paper. Finally, in the fourth section, we present novel results that have been obtained by applying the main theorems of the
paper.

\section{Preliminaries}\label{sec2}
In this section, we introduce some key concepts and definitions that enable us to precisely state our results.

Let $\Sigma \in \mathbb{R}^2$ be defined as $\Sigma = f^{-1}(0)$, where $f:\mathbb{R}^{2}\rightarrow \mathbb{R}$ is a 
smooth function with $0$ as a regular value (i.e., $\nabla f(p) \neq 0$ for any $p\in f^{-1}(0)$). Additionally, let $\Omega^r$ be the 
space of $C^r$ vector fields on $\mathbb{R}^2$ for $r\geqslant 1$. We consider the planar piecewise vector fields of the form
\begin{equation}\label{pwvf}
Z(x,y)= 
\begin{cases}
Z^{+}(x,y),&\text{for}~ (x,y)\in \Sigma^{+},\\
Z^{-}(x,y),&\text{for}~ (x,y)\in \Sigma^{-}, 
\end{cases} 
\end{equation}
where $Z^{\pm}\in \Omega^r$ and 
\begin{equation*}
\Sigma^{+}=\lbrace (x,y)\in \mathbb{R}^2 \mid f(x,y)> 0\rbrace,\quad
\Sigma^{-}=\lbrace (x,y)\in \mathbb{R}^2 \mid f(x,y)< 0\rbrace,
\end{equation*}
and the {\it switching boundary} $\Sigma$ corresponds to the boundary between these two regions.

The piecewise vector field $Z$ is said to be {\it continuous} if it satisfies $Z^{+} = Z^-$ on $\Sigma$. Otherwise, 
we say that it is {\it discontinuous}. A {\it crossing point} 
is defined as a point $p\in \Sigma$ where both vector fields are transversal to the switching boundary, and their normal components have 
the same sign. Thus, the {\it crossing region} is defined as
\begin{equation*}
\Sigma^{c}=\lbrace p\in \Sigma \mid Z^{+}f(p)Z^{-}f(p)>0\rbrace,
\end{equation*}
where $Z^{\pm}f(p)=\langle \bigtriangledown f(p), Z^{\pm}(p) \rangle$.
A {\it sliding/escaping point} $p$ is defined as a point on $\Sigma$ where both vector fields simultaneously point inward or outward 
from $\Sigma$, respectively. Then, the {\it sliding and escaping regions} are defined as
\begin{equation*}
\Sigma^{s}=\lbrace p\in \Sigma \mid Z^{+}f(p)<0, Z^{-}f(p)>0\rbrace,
\end{equation*}
and
\begin{equation*}
\Sigma^{e}=\lbrace p\in \Sigma \mid Z^{+}f(p)>0, Z^{-}f(p)<0\rbrace.
\end{equation*}

The solutions of the differential system $\dot{q} = Z(q)$ are the trajectories of $Z$, where the right-hand side of 
the equation is generally discontinuous. For basic concepts and results of ordinary differential equations with discontinuous right-hand 
side, see reference \cite{Fil}.

We here are interested in studying discontinuous systems that have the property $Z^{+}f(p)=Z^{-}f(p) $ for all $p\in \Sigma$, 
and such systems are known as {\it refracted systems}, see \cite{BMT,PRV}. We should note that in refracted systems, crossing regions 
are the only ones that exist.
\begin{figure}[] 
\centering
\includegraphics[scale=.8]{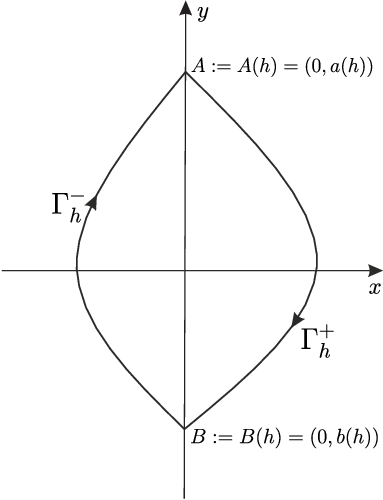}
\caption{The periodic orbit of system $\eqref{hamiep}|_{\varepsilon=0}$. \label{fig1}}
\end{figure}

Now, we present the first order Melnikov function of piecewise smooth differential systems.
Let us consider the piecewise polynomial near-Hamiltonian differential systems of the form
\begin{equation}\label{hamiep}
\left(\dot x, \dot y \right)= \begin{cases}
\left( H^{+}_{y}(x,y), -H^{+}_{x}(x,y)\right) +\varepsilon \left( p^{+}(x,y), q^{+}(x,y)\right), & x> 0,\\
\left( H^{-}_{y}(x,y), -H^{-}_{x}(x,y)\right) +\varepsilon \left( p^{-}(x,y), q^{-}(x,y)\right), & x< 0,
\end{cases} 
\end{equation}
where $p^{\pm}(x,y)$, $q^{\pm}(x,y)$ and $H^{\pm}(x,y)$ are real polynomials and $\varepsilon$
is a small real parameter. System \eqref{hamiep} can be separated into two analytic subsystems,
i.e. the right subsystem and the left subsystem, respectively,
\begin{equation}\label{hamiep1}
\begin{cases}
\dot x=H_{y}^{+}(x,y)+\varepsilon p^{+}(x,y), \\
\dot y=-H_{x}^{+}(x,y)+\varepsilon q^{+}(x,y), \tag{1a}
\end{cases}
\end{equation}
and
\begin{equation}\label{hamiep2}
\begin{cases}
\dot x=H_{y}^{-}(x,y)+\varepsilon p^{-}(x,y), \\
\dot y=-H_{x}^{-}(x,y)+\varepsilon q^{-}(x,y). \tag{1b}
\end{cases}
\end{equation}
We will suppose that $\eqref{hamiep}|_{\varepsilon=0}$ has a family of periodic orbits around the origin. For system 
$\eqref{hamiep}|_{\varepsilon=0}$, we make the following assumptions: \\
\textbf{Assumption (I).}\label{assumption1}
There exists an interval $\Omega=(h_{1},h_{2})$ and two points $A(h)=(0,a(h))$ and $B(h)=(0,b(h))$ such that for $h\in \Omega$,
\begin{equation*}
H^{+}(A(h))=H^{+}(B(h))=h,\qquad H^{-}(A(h))=H^{-}(B(h))=\tilde{h}, 
\end{equation*}
where $b(h)<0<a(h)$.\\
\textbf{Assumption (II).}
The subsystem $ \eqref{hamiep1}|_{\varepsilon=0} $ has an orbital arc $\Gamma_{h}^{+}$ starting from $A(h)$ and ending at $B(h)$
defined by $H^{+}(x,y)=h$, $x> 0$; the subsystem $ \eqref{hamiep2}|_{\varepsilon=0} $ has an orbital 
arc $\Gamma_{h}^{-}$ starting from $B(h)$ and ending at $A(h)$ defined by $H^{-}(x,y)=H^{-}(B(h))=\tilde{h}$, $x<0$.

Under the above assumptions, the unperturbed system $\eqref{hamiep}\vert_{\varepsilon=0}$ has a family of non-smooth 
periodic orbits $ \Gamma_{h}=\Gamma_{h}^{+}\cup \Gamma_{h}^{-} $, $h\in \Omega$. For definiteness, we assume that the 
orbits $\Gamma_{h}$ for $h\in \Omega$ is orientated clockwise; see Figure \ref{fig1}. The authors in \cite{LH}
defined a bifurcation function $F(h, \varepsilon)$ for system \eqref{hamiep}, where $F(h,0)=I(h)$.
Given Assumptions (I) and (II), the first order Melnikov function of system \eqref{hamiep} was derived by them as follows:
\begin{equation}\label{Melnikov}
I(h):=\dfrac{H_{y}^{+}(A)}{H_{y}^{-}(A)}\left[ \dfrac{H_{y}^{-}(B)}{H_{y}^{+}(B)}\int_{\Gamma_{h}^{+}}q^{+}dx-
p^{+}dy+\int_{\Gamma_{h}^{-}}q^{-}dx-p^{-}dy\right],\qquad h\in \Omega.
\end{equation}

Finally, we introduce the notion of Chebyshev systems. 
The reader is referred to \cite{GMV,MV} for more details on the following definitions and the next results.
\begin{definition}
Let $f_{0},f_{1},\ldots,f_{n-1}$ be real analytic functions on some
open interval $\mathbb{I}$ of $\mathbb{R}$. Then
\begin{enumerate}
\item[(i)] The set $\lbrace f_{0},f_{1},\ldots,f_{n-1}\rbrace$ is called a Chebyshev system (for short, a T-system) on
$\mathbb{I}$ if any nontrivial linear combination
\begin{align*}
\alpha_{0}f_{0}(x)+\alpha_{1}f_{1}(x)+\ldots+\alpha_{n-1}f_{n-1}(x)
\end{align*}
has at most $n-1$ isolated zeros on $\mathbb{I}$.
\item[(ii)] The set $\lbrace f_{0},f_{1},\ldots,f_{n-1}\rbrace$ is called a complete Chebyshev
system (for short, a CT-system) on $\mathbb{I}$ if for all $k=1,2,\ldots,n$, any nontrivial
linear combination
\begin{align*}
\alpha_{0}f_{0}(x)+\alpha_{1}f_{1}(x)+\ldots +\alpha_{k-1}f_{k-1}(x)
\end{align*}
has at most $k-1$ isolated zeros on $\mathbb{I}$.
\item[(iii)] The ordered set $\lbrace f_{0},f_{1},\ldots,f_{k-1}\rbrace$ is called an extended complete
Chebyshev system (for short, an ECT-system) on $\mathbb{I}$ if for all $k=1,2,...,n$, any
nontrivial linear combination
\begin{align*}
\alpha_{0}f_{0}(x)+\alpha_{1}f_{1}(x)+\ldots+\alpha_{k-1}f_{k-1}(x)
\end{align*}
has at most $k-1$ isolated zeros on $\mathbb{I}$ counting multiplicity.
\end{enumerate}
\end{definition}

\begin{definition}
Let $f_{0},f_{1},\ldots,f_{k-1}$ be real analytic functions on some open 
interval $\mathbb{I}$ of $\mathbb{R}$. The {\it continuous Wronskian} of $ \lbrace f_{0},f_{1},\ldots,f_{k-1} \rbrace $ at $x\in \mathbb{I}$ is
\begin{equation*} 
W[\bold{f_k}](x)= \det
\begin{bmatrix} 
f_0(x) & f_1(x) & \cdots & f_{k-1}(x) \\ 
f^{\prime}_0(x) & f^{\prime}_1(x) & \cdots & f^{\prime}_{k-1}(x) \\ 
\vdots & \vdots & \ddots & \vdots \\ 
f^{(k-1)}_0(x) & f^{(k-1)}_1(x) & \cdots & f^{(k-1)}_{k-1}(x) 
\end{bmatrix}.
\end{equation*}
The {\it discrete Wronskian} of $\lbrace f_{0},f_{1},\ldots,f_{k-1} \rbrace $ at $ (x_{0},x_{1},\ldots,x_{k-1})\in \mathbb{I}^k$
is
\begin{equation*}
D[\bold{f_k}](\bold{x_k})= \det 
\begin{bmatrix} 
f_0(x_0) & f_1(x_0) & \cdots & f_{k-1}(x_0) \\ 
f_0(x_1) & f_1(x_1) & \cdots & f_{k-1}(x_1) \\ 
\vdots & \vdots & \ddots & \vdots \\
f_0(x_{k-1}) & f_1(x_{k-1}) & \cdots & f_{k-1}(x_{k-1}) 
\end{bmatrix}.
\end{equation*}
\end{definition}

Note that in the above definitions we used the notation 
\begin{equation*}
\bold{f_k}=f_{0},f_{1},\ldots,f_{k-1}\quad \text{and}\quad \bold{x_k}=x_{0},x_{1},\ldots,x_{k-1}.
\end{equation*}
\begin{lemma}\label{lem4}(\cite{KS,Ma}) The following statements hold:
\begin{itemize}
\item[(a)] The set of functions $\lbrace f_{0},f_{1},\ldots,f_{n-1}\rbrace$ is a CT-system on $\mathbb{I}$ if and only if, for each
$k=1,2,\ldots,n$,
\begin{equation*}
D[\bold{f_k}](\bold{x_k})\neq 0,\;\;\ \text{for all}\;\ \bold{x_k}\in \mathbb{I}^{k}\;\ \text{such that}\;\ x_{i}\neq x_{j}\;\ \text{for}
\;\ i\neq j.
\end{equation*}
\item[(b)]The set of functions $\lbrace f_{0},f_{1},\ldots,f_{n-1}\rbrace$ is an ECT-system on $\mathbb{I}$ if and only if, for each
$k=1,2,\ldots,n$,
\begin{align*}
W[\bold{f_k}](x)\neq 0,\;\;\ \text{for all}\;\ x\in \mathbb{I}.
\end{align*}
\end{itemize}
\end{lemma}
\section{Proof of the main results}\label{sec3}
In this section we will prove Theorems \ref{the1} and \ref{the2}. 
Let us begin by assuming that $\chi_{1}(x)=\chi_{2}(x)=1$ in the Hamiltonian function \eqref{Hm}. In this case, we have
\begin{equation}\label{hamm1}
H(x,y)=\begin{cases}
H^{+}(x,y)=y^{2}+\Psi_{1}(x), & x> 0,\\
H^{-}(x,y)=y^{2}+\Psi_{2}(x), & x< 0, 
\end{cases}
\end{equation}
and also the corresponding perturbed Hamiltonian system is given by
\begin{equation}\label{Ham1}
\left(\dot x, \dot y \right)= \begin{cases}
\left(2 y, -\Psi_{1}^{\prime}(x)\right) +\varepsilon \left( p^{+}(x,y), q^{+}(x,y)\right), & x> 0,\\
\left(2 y, -\Psi_{2}^{\prime}(x)\right) +\varepsilon \left( p^{-}(x,y), q^{-}(x,y)\right), & x< 0,
\end{cases}
\end{equation}
where $0<|\varepsilon|\ll 1$ and
\begin{equation*}
p^{\pm}(x,y)=\sum_{i+j=0}^{n}a^{\pm}_{ij}x^{i}y^{j},\qquad q^{\pm}(x,y)=\sum_{i+j=0}^{n}b^{\pm}_{ij}x^{i}y^{j}.
\end{equation*}
The assumptions (\ref{H1}) on $H$ implies both Assumptions $\text{(I)}$ and $\text{(II)}$. Therefore formula \eqref{Melnikov} for system \eqref{Ham1} changes to 
the following form
\begin{equation}\label{Melnk1}
I(h)=\int_{\Gamma_{h}^{+}}q^{+}(x,y)dx-p^{+}(x,y)dy+\int_{\Gamma_{h}^{-}}q^{-}(x,y)dx-
p^{-}(x,y)dy,\qquad h\in \Omega ,
\end{equation}
with
\begin{equation*}
\Gamma^{+}_{h}=\lbrace(x,y)\in \mathbb{R}^{2}\vert H^{+}(x,y)=h, x> 0\rbrace,\qquad
\Gamma^{-}_{h}=\lbrace(x,y)\in \mathbb{R}^{2}\vert H^{-}(x,y)=h, x< 0\rbrace.
\end{equation*}
Note that $\tilde{h}$ is equal to $h$, as $\Psi_1(0)$ and $\Psi_2(0)$ both equal zero.
Now, we will first obtain the algebraic structure of Melnikov function $I(h)$ for systems \eqref{hamm1}. 
\begin{lemma}
Assuming that the function $H(x,y)$ defined in \eqref{hamm1} has a family of ovals $\Gamma_{h}^{+}\cup\Gamma_{h}^{-}$,
we can express $I(h)$ for $h\in (0,h_0)$ in the system \eqref{Ham1} as follows
\begin{equation}\label{Melnk2}\displaystyle
I(h)=\begin{cases}
\tilde{a}_{01}\displaystyle\int_{\Gamma^{+}_{h}}ydx+\tilde{b}_{01}\displaystyle\int_{\Gamma^{-}_{h}}ydx,&n=1,\\
\\
\displaystyle\int_{\Gamma^{+}_{h}}p_{1}(x)ydx+\displaystyle\int_{\Gamma^{-}_{h}}q_{1}(x)ydx,&n=2,\\
\\
\displaystyle\int_{\Gamma^{+}_{h}}p_{k_1}(x,h)ydx+\displaystyle\int_{\Gamma^{-}_{h}}q_{k_2}(x,h)ydx,& n\geq 3,
\end{cases}
\end{equation} 
where $p_{1}$ and $q_{1}$ are linear functions in $x$, and $p_{k_1}(x, h)$ and $q_{k_2}(x, h)$ are polynomials in $x$ and $h$ 
of degree $k_{i} = \dfrac{m_{i}(n-1)}{2}$ if $n$ is odd and $k_{i} = \dfrac{m_{i}(n-2)}{2}+1$ if $n$ is even, where 
$m_{i}=\operatorname{deg}(\Psi_{i}), i=1,2$.
\end{lemma}
\begin{proof}
Let us assume that the clockwise closed orbit $ \Gamma_{h}^{+}\cup \Gamma_{h}^{-} $ intersects the $y$-axis at two distinct points $A$ 
and $B$, and let $D$ denote the region bounded by $\Gamma^{+}_{h}\cup \overrightarrow{BA} $. Using Green's theorem
and the fact that $\int_{\overrightarrow{BA}}x^{i-1}y^{j+1}dx=0$ for $i\geq 1$, it follows that
\begin{equation*}
\int_{\Gamma^{+}_{h}}x^{i-1}y^{j+1}dx=\oint_{\Gamma^{+}_{h}\cup \overrightarrow{BA} 
}x^{i-1}y^{j+1}dx=(j+1)\int\int_{D}x^{i-1}y^{j}dxdy,
\end{equation*}
and 
\begin{equation*}
\int_{\Gamma^{+}_{h}}x^{i}y^{j}dy=\oint_{\Gamma^{+}_{h}
\cup \overrightarrow{BA}}x^{i}y^{j}dy=
-i\int\int_{D}x^{i-1}y^{j}dxdy.
\end{equation*}
Hence, for $i\geq 1$, we get that
\begin{equation}\label{I1}
\int_{\Gamma^{+}_{h}}x^{i}y^{j}dy=-\frac{i}
{j+1}\int_{\Gamma^{+}_{h}}x^{i-1}y^{j+1}dx.
\end{equation}
Similarly, for $i\geq 1$, we can obtain that
\begin{equation}\label{I2}
\int_{\Gamma^{-}_{h}}x^{i}y^{j}dy=-\frac{i}
{j+1}\int_{\Gamma^{-}_{h}}x^{i-1}y^{j+1}dx.
\end{equation}
Then, using \eqref{I1} and \eqref{I2}, the line integral \eqref{Melnk1} can be written as
\begin{equation*}
\begin{aligned}
I(h) =&
\int_{\Gamma_{h}^{+}}\sum_{i+j=0}^{n}b^{+}_{ij}x^{i}y^{j}dx
-\int_{\Gamma_{h}^{+}}\sum_{i+j=0}^{n}a^{+}_{ij}x^{i}y^{j}dy+\int_{\Gamma_{h}^{-}}\sum_{i+j=0}^{n}b^{-}_{ij}x^{i}y^{j}dx
-\int_{\Gamma_{h}^{-}}\sum_{i+j=0}^{n}a^{-}_{ij}x^{i}y^{j}dy\\
=&\sum_{i+j=0}^{n}b^{+}_{ij}\int_{\Gamma_{h}^{+}}x^{i}y^{j}dx
+\sum_{i+j=0}^{n}\frac{i}{j+1}a^{+}_{ij}\int_{\Gamma^{+}_{h}}x^{i-1}y^{j+1}dx\\
&+\sum_{i+j=0}^{n}b^{-}_{ij}\int_{\Gamma_{h}^{-}}x^{i}y^{j}dx
+\sum_{i+j=0}^{n}\frac{i}{j+1}a^{-}_{ij}\int_{\Gamma^{-}_{h}}x^{i-1}y^{j+1}dx\\
=&\sum_{i+j=0,i\geq 0, j\geq 1}^{n}\tilde{a}_{ij} I_{ij}(h)
+\sum_{i+j=0,i\geq 0, j\geq 1}^{n}\tilde{b}_{ij} J_{ij}(h),
\end{aligned}
\end{equation*}
where
\begin{equation*}
\begin{aligned}
&
\tilde{a}_{ij}=b^{+}_{ij}+\frac{i+1}{j}a^{+}_{(i+1)(j-1)},\\
&\tilde{b}_{ij}=b^{-}_{ij}+\frac{i+1}{j}a^{-}_{(i+1)(j-1)},\\
&
I_{ij}(h)=\int_{\Gamma_{h}^{+}}x^{i}y^{j}dx,\qquad
J_{ij}(h)=\int_{\Gamma_{h}^{-}}x^{i}y^{j}dx.
\end{aligned}
\end{equation*}
On the other hand, considering the orbital arc $\Gamma_{h}^{+}$, we can observe that the equation $y^{2}=h-\Psi_{1}(x)$ holds. Therefore, for any value of 
$l\geq0$, we can deduce that
\begin{equation*}
\int_{\Gamma_{h}^{+}}x^{i}y^{j}dx=\oint_{\Gamma_{h}^{+}\cup \overrightarrow{BA}}x^{i}y^{j}dx=
\begin{cases}
0,& j=2l,\\
\displaystyle\int_{\Gamma_{h}^{+}}x^{i}\left( h-\Psi_{1}(x)\right)^{l} ydx, & j=2l+1,
\end{cases}
\end{equation*}
similarly, along the orbital arc $\Gamma_{h}^{-}$, we can see that the equation $y^{2}=h-\Psi_{2}(x)$ holds. Consequently,
for $l\geq0$, we can conclude that
\begin{equation*}
\int_{\Gamma_{h}^{-}}x^{i}y^{j}dx=\oint_{\Gamma_{h}^{-}\cup \overrightarrow{AB}}x^{i}y^{j}dx=
\begin{cases}
0,& j=2l,\\
\displaystyle\int_{\Gamma_{h}^{-}}x^{i}\left( h-\Psi_{2}(x)\right)^{l} ydx, & j=2l+1.
\end{cases}
\end{equation*}
Thus, using the above information, the statements of the lemma immediately follow.
\end{proof}
\begin{figure}[]
\centering
\includegraphics[scale=.85]{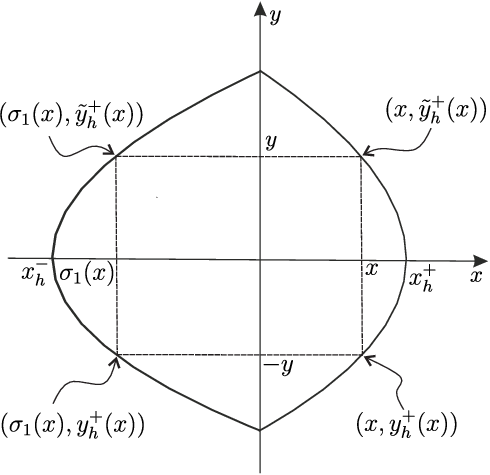}
\caption{Notation related to piecewise smooth oval around the origin. \label{fig2}}
\end{figure}

We often consider that $p^{\pm}(x,y)$ and $q^{\pm}(x,y)$ belong to a specific function space. For instance, 
we can consider $p^{\pm}(x,y)=y\,p_{\pm}(x)$ and $q^{\pm}(x,y)=y\,q_{\pm}(x)$, which can be rescaled to transform 
integral \eqref{Melnk2} into the following form
\begin{equation}\label{LC}
a^{+}_{0}I_{0}(h)+a^{-}_{0}J_{0}(h)+a^{+}_{1}I_{1}(h)+a^{-}_{1}J_{1}(h)+\ldots+a^{+}_{n-1}I_{n-1}(h)
+a^{-}_{n-1}J_{n-1}(h),
\end{equation}
where the constants $a^{\pm}_0,a^{\pm}_1,\ldots,a^{\pm}_{n-1}$ depend on the initial parameters, and
\begin{equation*}
I_{i}(h)=\int_{\Gamma_{h}^{+}}x^{i} y dx,\quad J_{i}(h)=\int_{\Gamma_{h}^{-}}x^{i} y dx,\qquad i=0,1,\ldots, n-1.
\end{equation*}

The set of functions $\left\lbrace I_{0}(h),J_{0}(h),I_{1}(h),J_{1}(h)\ldots,I_{n-1}(h),J_{n-1}(h)\right\rbrace $ may not always be 
linearly independent. Nevertheless, it is possible to identify the maximal subset of these functions that is linearly independent by applying 
linear algebra techniques. The maximal subset of linearly independent functions that we can obtain from the family of functions 
\begin{equation*}
\left\lbrace I_{0}(h),J_{0}(h),I_{1}(h),J_{1}(h)\ldots,I_{n-1}(h),J_{n-1}(h)\right\rbrace,
\end{equation*}
will consist of $I_{0}(h),I_{1}(h),\ldots,I_{n-1}(h)$ and 
some or all of the functions $J_{i}(h)$ that are linearly independent of the functions $I_{0}(h),I_{1}(h),\ldots,I_{n-1}(h)$ and the
previously chosen $J_{j}(h)$, for $j < i$.

Let $m$ be the number of functions in this subset other than the $I_{i}$ functions. Then, the subset can be written as
\begin{equation*}
\left\lbrace I_{0}(h),I_{1}(h),\ldots,I_{n-1}(h),J_{i_1}(h),J_{i_2}(h),\ldots,J_{i_m}(h)\right\rbrace,
\end{equation*} 
where $i_1 < i_2 < \cdots < i_m$ are the indices of the $J$ functions that are included in the subset.
Note that $m$ can be any integer between $0$ and $n$. If none of the $J_{i}$ functions are linearly independent of the $I_{i}$ functions,
then $m=0$ and the subset consists only of the $I_{i}$ functions. If all of the $J_{i}$ functions are linearly independent of the 
$I_{i}$ functions, then $m=n$ and the subset consists of all $I_{i}$ functions and all $J_{i}$ functions. In what follows, along the curve $ \Gamma_{h} $, we consider 
that the branches of $ \Gamma^{\pm}_{h} $ are defined by
\begin{equation*}
\begin{aligned}
\tilde{y}_{h}^{+}(x)&:=\sqrt{h-\Psi_{1}(x)},\qquad y_{h}^{+}(x):=-\sqrt{h-\Psi_{1}(x)},\\
\tilde{y}_{h}^{-}(x)&:=\sqrt{h-\Psi_{2}(x)},\qquad y_{h}^{-}(x):=-\sqrt{h-\Psi_{2}(x)},
\end{aligned}
\end{equation*}
for each $h\in (0,h_{0})$. We also note that $\tilde{y}_{h}^{+}(x)=\tilde{y}_{h}^{-}(\sigma_{1}(x))$ and 
$y_{h}^{+}(x)=y_{h}^{-}(\sigma_{1}(x))$ for $x\in (0,x_{r})$ (see Figure \ref{fig2}). 
\begin{proof}[Proof of Theorem \ref{the1}]
Now we will prove the first theorem of the paper.
We will compute the derivatives of both $I(h)$ and $J(h)$. These derivatives are essential to calculate the Wronskian of the family 
of functions.
\begin{lemma}\label{lem1}
Assume that $f_1$ and $f_2$ are analytic functions on the intervals $(0, x_r)$ and $(x_l, 0)$ respectively, and let us consider 
\begin{equation*}
I(h)=\int_{\Gamma_{h}^{+}}f_{1}(x)y^{2s-1} dx,\quad J(h)=\int_{\Gamma_{h}^{-}}f_{2}(x)y^{2s-1} dx,
\end{equation*}
then if $s>2(n-1)$, then we get that
\begin{equation*}
I^{(k)}(h)=c_{k}\int_{0}^{x_{h}^{+}}f_{1}(x) \tilde{y}_{h}^{+}(x)^{2(s-k)-1} dx, \qquad k=0,1,\ldots,2n-1,
\end{equation*}
and 
\begin{equation*}
J^{(k)}(h)=-c_{k}\int_{0}^{x_{h}^{+}}f_{2}(\sigma_{1}(x))\sigma_{1}^{\prime}(x) \tilde{y}_{h}^{+}(x)^{2(s-k)-1} dx, \qquad 
k=0,1,\ldots,2n-1,
\end{equation*}
where $c_{k}:=2(2s-1)(2s-3)\ldots(2(s-k)+1)$.
\end{lemma}
\begin{proof}
To prove the above lemma by induction on $k$, we first show that the statement is true for $k=0$. Let us assume that the non-smooth 
oval $\Gamma_{h}^{+}\cup \Gamma_{h}^{-}$ has a clockwise orientation. By using the fact that 
$y_{h}^{\pm}(x)=-\tilde{y}_{h}^{\pm}(x)$, we can obtain that
\begin{equation*}
\begin{aligned}
I(h)=&\; \int_{0}^{x_{h}^{+}}f_{1}(x){\tilde{y}_{h}^{+}}(x)^{2s-1} dx+\int_{x_{h}^{+}}^{0}f_{1}(x)y_{h}^{+}(x)^{2s-1} dx\\
=&\;2 \int_{0}^{x_{h}^{+}}f_{1}(x){\tilde{y}_{h}^{+}}(x)^{2s-1}dx,
\end{aligned}
\end{equation*}
and 
\begin{equation*}
\begin{aligned}
J(h)=&\;\int_{0}^{x_{h}^{-}}f_{2}(x)y_{h}^{-}(x)^{2s-1} dx+\int_{x_{h}^{-}}^{0}f_{2}(x)\tilde{y}_{h}^{-}(x)^{2s-1} dx\\
=&\; 2\int_{x_{h}^{-}}^{0}f_{2}(x)\tilde{y}_{h}^{-}(x)^{2s-1} dx\\
=&\; 2\int_{x_{h}^{+}}^{0}f_{2}(\sigma_{1}(u))\sigma_{1}^{\prime}(u)y^{2s-1}
\Big{\vert}_{\tilde{y}_{h}^{-}(\sigma_{1}(u))} du,
\end{aligned}
\end{equation*}
where in the last equality we applied the change of variable $x = \sigma_{1}(u)$. Also we have 
$\tilde{y}_{h}^{-}(\sigma_{1}(u))=\tilde{y}_{h}^{+}(u)$, and the above expression leads to
\begin{equation*}
J(h)=-2\int_{0}^{x_{h}^{+}}f_{2}(\sigma_{1}(x))\sigma_{1}^{\prime}(x){\tilde{y}_{h}^{+}(x)}^{2s-1} dx.
\end{equation*}
So the proof holds for $k=0$. Assuming that the proof holds for values of $k < 2n-1$, we can now use the fact that $s>2(n-1)$ to obtain that
\begin{equation*}
\begin{aligned}
I^{(k+1)}(h)=&\dfrac{d}{dh}\int_{0}^{x_{h}^{+}}c_{k}f_{1}(x) \tilde{y}_{h}^{+}(x)^{2(s-k)-1} dx\\
=&c_{k}f_{1}(x_{h}^{+}) \tilde{y}_{h}^{+}(x_{h}^{+})^{2(s-k)-1}\dfrac{dx_{h}^{+}}{dh}
+\int_{0}^{x_{h}^{+}}c_{k}f_{1}(x) \dfrac{d\tilde{y}_{h}^{+}(x)^{2(s-k)-1}}{dh} dx\\
=&\int_{0}^{x_{h}^{+}}c_{k}\left( 2(s-k)-1\right) f_{1}(x) \tilde{y}_{h}^{+}(x)^{2(s-k)-2} dx\\
=&\int_{0}^{x_{h}^{+}}c_{k+1}f_{1}(x) \tilde{y}_{h}^{+}(x)^{2(s-(k+1))-1} dx.
\end{aligned}
\end{equation*}
We note that in the second equality we use the fact that $ \tilde{y}_{h}^{+}(x)$ is equal to zero at $x=x_{h}^{+}$.
Therefore, the statement is true for $k+1$ as well. 

Since the proof of the derivative of $J(h)$ up to $2n-1$ is similar to the derivative of $I(h)$, we will not provide it here. 
Thus, we have completed the proof of the lemma.
\end{proof}

Let's now consider the line integrals of the form
\begin{equation*}
I_{i}(h)=\int_{\Gamma_{h}^{+}}f_{1i}(x)y^{2s-1} dx,\quad
J_{i}(h)=\int_{\Gamma_{h}^{-}}f_{2i}(x)y^{2s-1} dx,
\end{equation*}
where $f_{1i}$ and $f_{2i}$ are analytic functions on $(0, x_r)$ and $(x_l, 0)$, respectively. To simplify our approach, we rename this 
family of functions as 
\begin{equation*} 
L_i (h):=
\begin{cases}
I_\frac{i}{2}(h), & i = 0, 2, \ldots, 2n-2, \\ 
J_\frac{i-1}{2}(h), & i = 1, 3, \ldots, 2n-1,
\end{cases} 
\end{equation*} 
and, using Lemma \ref{lem1}, the derivatives of the functions $L_i (h)$ can be given as
\begin{equation*} 
L^{(k)}_i (h)=c_{k}\int_{0}^{x_{h}^{+}}l_{i}(x) \tilde{y}_{h}^{+}(x)^{2(s-k)-1},\quad i=0,1,\ldots, 2n-1,
\end{equation*} 
where $ k=0,1,\ldots,2n-1 $, and 
\begin{equation*} 
l_i (x)=
\begin{cases} 
f_{1\left( \frac{i}{2}\right) }(x), & i = 0, 2, \ldots, 2n-2, \\ 
-f_{2\left(\frac{i-1}{2}\right)}(\sigma_{1}(x))\sigma_{1}^{\prime}(x), & i = 1, 3, \ldots, 2n-1.
\end{cases} 
\end{equation*} 
\begin{proposition}\label{pr1}
Suppose that $s > 2(n-1)$. Then, for each $k = 1, 2, \ldots, 2n$, the Wronskian of $(L_0,L_1,\ldots,L_{k-1})$ 
at $h \in (0, h_0)$ is given by
\begin{equation*}
W[\bold{L}_\bold{k}](h)
=m_{k-1} \int_{0}^{x^{+}_{h}}\ldots\int_{0}^{x^{+}_{h}} D[\bold{l}_{\bold{k}}](\bold{x_{k}})\prod_{i=0}^{k-1}
y_{i}^{2(s-i)-1}dx_{0}dx_{1}\ldots dx_{k-1},
\end{equation*}
where $y_{i}=\tilde{y}_{h}^{+}(x_{i})$ and $m_{k-1}=\prod_{i=0}^{k-1}c_{i}$.
\end{proposition}
\begin{proof}
To calculate the Wronskian of the function family $(L_0,L_1,\ldots,L_{k-1})$, where $k$ varies from $1$ to $2n$, we use the Leibniz formula, which is defined as follows:
\begin{equation*}
\begin{aligned}
W[\bold{L}_\bold{k}](h)=&det \left( L_{j}^{(i)}(h)\right) _{0\leq i,j \leq k-1}= \sum_{\pi \in S_{k}} \operatorname{sgn}(\pi)
\prod_{i=0}^{k-1}L^{(i)}_{\pi(i)}(h)\\
=& \sum_{\pi \in S_{k}} \operatorname{sgn}(\pi)\prod_{i=0}^{k-1}c_{i}\int_{0}^{x_{h}^+}l_{\pi(i)}(x)
\tilde{y}_{h}^{+}(x)^{2(s-i)-1}
dx\\=& \sum_{\pi \in S_{k}} \operatorname{sgn}(\pi)
\left( \prod_{i=0}^{k-1}c_{i}\right) \int_{0}^{x_{h}^+}\ldots \int_{0}^{x_{h}^+}l_{\pi(i)}(x_{i})\prod_{i=0}^{k-1}
\tilde{y}_{h}^{+}(x)^{2(s-i)-1}(x_{i})dx_{0}\ldots dx_{k-1}\\
=& \prod_{i=0}^{k-1}c_{i}
\int_{0}^{x_{h}^+}\ldots \int_{0}^{x_{h}^+}\sum_{\pi \in S_{k}} \operatorname{sgn}(\pi)\prod_{i=0}^{k-1}
l_{\pi(i)}(x_{i})\prod_{i=0}^{k-1}\tilde{y}_{h}^{+}(x)^{2(s-i)-1}(x_{i})dx_{0}\ldots dx_{k-1}\\
=& \prod_{i=0}^{k-1}c_{i}
\int_{0}^{x_{h}^+}\ldots \int_{0}^{x_{h}^+}D[\bold{l}_{\bold{k}}](\bold{x_{k}})\prod_{i=0}^{k-1}\tilde{y}_{h}^{+}(x)^{2(s-i)-1}
(x_{i})dx_{0}\ldots dx_{k-1},
\end{aligned}
\end{equation*}
where $S_{k}$ is the set of all permutations of $0, 1, ..., k-1$, and $\operatorname{sgn}(\pi)$ 
is the sign of the permutation $\pi$. The proof of the lemma is now complete.
\end{proof}
From Proposition \ref{pr1}, we get that 
\begin{equation*}
W[\bold{L}_\bold{k}](h)
=m_{k-1} \int\ldots\int_{\bigtriangleup_{k}(h)} D[\bold{l}_{\bold{k}}](\bold{x_{k}})\prod_{i=0}^{k-1}y_{i}^{2(s-i)-1}dx_{0}dx_{1}\ldots dx_{k-1},
\end{equation*}
where $m_{k-1}\neq 0$, and by assumption the family of functions $(l_{0},l_{1},\ldots,l_{2n-1})$ is a CT-system on $(0,x_{r})$, and it implies that 
$ W[\bold{L}_\bold{k}](h)\neq 0 $. Therefore, Theorem \ref{the1}
was proved for $\chi_{1}(x)=\chi_{2}(x)=1$. 

Since the functions $\chi_1(x)$ and $\chi_2(x)$ take positive values in the intervals $(0,x_r)$ and $(x_l,0)$, respectively, and 
$\lim_{x\to 0^+} \chi_1(x) = \lim_{x\to 0^-} \chi_2(x) > 0$, we can define a new coordinate system 
$(u,v) = \varphi_i(x,y) = (x, \sqrt{\chi_i(x)} y)$, for $i=1,2$. Using this transformation, we obtain 
\begin{equation*}
\gamma^{+}_{h}:=\varphi^{-1}_{1}(\Gamma^{+}_{h})\subset\lbrace v^{2}+\Psi_{1}(u)=h, u>0\rbrace,\qquad
\gamma^{-}_{h}:=\varphi^{-1}_{2}(\Gamma^{-}_{h})\subset\lbrace v^{2}+\Psi_{2}(u)=h, u>0\rbrace.
\end{equation*}
and the integrals $I_i(h)$ and $J_i(h)$ can be expressed as
\begin{equation*}
I_{i}(h)=\int_{\Gamma_{h}^{+}}f_{1i}(x)y^{2s-1} dx=\int_{\gamma_{h}^{+}}\left( \dfrac{f_{1i}}{\chi_{1}^{\frac{2s-1}{2}}}\right) 
(u)v^{2s-1} du,
\end{equation*}
and 
\begin{equation*}
J_{i}(h)=\int_{\Gamma_{h}^{-}}f_{2i}(x)y^{2s-1} dx=\int_{\gamma_{h}^{-}}\left( \dfrac{f_{2i}}{\chi_{2}^{\frac{2s-1}{2}}}\right) 
(u)v^{2s-1} du.
\end{equation*}
If we use the same notation as in the case $\chi_1(x) = \chi_2(x) = 1$, that is, 
\begin{equation*}
f_{1i}:=\dfrac{f_{1i}}{\chi_{1}^{\frac{2s-1}{2}}}, \quad f_{2i}:= \dfrac{f_{2i}}{\chi_{2}^{\frac{2s-1}{2}}}\quad\text{and}\quad
(x,y):=(u,v),
\end{equation*}
then the proof of Theorem \ref{the1} is complete.
\end{proof}
\begin{proof}[Proof of Theorem \ref{the2}]
In order to compute the Wronskian of the family of functions in Theorem \ref{the2}, we must take the derivative of $I(h)$ up to 
$n-1$. Consequently, we only need to consider $s>n-2$, and the derivative of $I(h)$ will be identical to that in Lemma \ref{lem1} as
\begin{equation*}
I^{(k)}(h)=c_{k}\int_{0}^{x_{h}^{+}}f(x) \tilde{y}_{h}^{+}(x)^{2(s-k)-1} dx, \qquad k=0,1,\ldots,n-1.
\end{equation*}
\begin{proposition}\label{pro2}
Assume that $s>n-2$. Then, for each $k = 1, 2,\ldots, n$, the Wronskian of 
$(I_{0}, I_{1}, \dots, I_{k-1})$ at $h\in(0, h_{0})$ is given by 
\begin{equation*}
W[\bold{I}_{\bold{k}}](h)
=m_{k-1} \int_{0}^{x^{+}_{h}}\ldots\int_{0}^{x^{+}_{h}} D[\bold{l}_{\bold{k}}](\bold{x_{k}})\prod_{i=0}^{k-1}y_{i}^{2(s-i)-1}dx_{0} dx_{1}\ldots dx_{k-1},
\end{equation*}
where $y_{i}=\tilde{y}_{h}^{+}(x_{i})$, $m_{k-1}=\prod_{i=0}^{k-1}c_{i}$, $l_{i}(x):= 
f_{i}(x)$.
\end{proposition}
\begin{proof}
The Wronskian of the family of functions $(I_0, I_1,\dots, I_{k-1})$ at $h\in(0, h_{0})$, where $k=1,\dots, n-1$, is defined as
\begin{equation*}
\begin{aligned}
W[\bold{I}_{\bold{k}}](h)=&det \left( I_{j}^{(i)}(h)\right) _{0\leq i,j \leq k-1}= \sum_{\pi \in S_{k}} \operatorname{sgn}(\pi)
\prod_{i=0}^{k-1}I^{(i)}_{\pi(i)}(h)\\
=& \sum_{\pi \in S_{k}} \operatorname{sgn}(\pi)\prod_{i=0}^{k-1}c_{i}\int_{0}^{x_{h}^+}l_{\pi(i)}(x)\tilde{y}_{h}^{+}(x)^{2(s-i)-1}
dx\\=& \sum_{\pi \in S_{k}} \operatorname{sgn}(\pi)\left( \prod_{i=0}^{k-1}c_{i}\right) 
\int_{0}^{x_{h}^+}\ldots \int_{0}^{x_{h}^+}l_{\pi(i)}(x_{i})\prod_{i=0}^{k-1}\tilde{y}_{h}^{+}(x)^
{2(s-i)-1}(x_{i})dx_{0}\ldots dx_{k-1}\\
=& \prod_{i=0}^{k-1}c_{i}
\int_{0}^{x_{h}^+}\dots \int_{0}^{x_{h}^+}\sum_{\pi \in S_{k}} \operatorname{sgn}(\pi)\prod_{i=0}^{k-1}
l_{\pi(i)}(x_{i})\prod_{i=0}^{k-1}\tilde{y}_{h}^{+}(x)^{2(s-i)-1}(x_{i})dx_{0}\ldots dx_{k-1}\\
=& \prod_{i=0}^{k-1}c_{i}
\int_{0}^{x_{h}^+}\ldots \int_{0}^{x_{h}^+}D[\bold{l}_{\bold{k}}](\bold{x_{k}})\prod_{i=0}^{k-1}\tilde{y}_{h}^{+}(x)^{2(s-i)-1}
(x_{i})dx_{0}\ldots dx_{k-1}.
\end{aligned}
\end{equation*}
The set $S_{k}$ is defined as the set of all permutations of ${0, 1, ..., k-1}$. The function $\operatorname{sgn}(\pi)$ gives the sign of 
a permutation $\pi$. With this, we have completed the proof of the lemma.
\end{proof}
Now, from Proposition \ref{pro2}, it follows that
\begin{equation*}
W[\bold{I}_{\bold{k}}](h)
=m_{k-1} \int_{0}^{x^{+}_{h}}\ldots\int_{0}^{x^{+}_{h}} D[\bold{l}_{\bold{k}}](\bold{x_{k}})\prod_{i=0}^{k-1}y_{i}^{2(s-i)-1}dx_{0} dx_{1}\ldots dx_{k-1},
\end{equation*}
where $m_{k-1} \neq 0$ and, by assumption, the family of functions $(l_0, l_1, \dots, l_{n-1})$ is a CT-system on $(0, x_r)$. This 
implies that $W[\bold{I}_{\bold{k}}](h) \neq 0$. 
Therefore, Theorem \ref{the2} has been proved for $\chi_1(x) = \chi_2(x) = 1$. The proof for the 
general case is similar to that of Theorem \ref{the1}, we omit the details for brevity.
\end{proof}
\section{Applications}
In this section, we first begin by introducing a lemma that enables us to apply the main theorems of the paper when $s>2(n-1)$ and 
$s>n-2$ assumptions are not satisfied for Theorems \ref{the1} and \ref{the2}, respectively.
\begin{lemma}\label{lems}
Let $\Gamma_{h}^{+}\cup \Gamma_{h}^{-}$ be a non-smooth oval around the origin, and we consider a function $F$ 
such that $F/{\Psi_{i}}^{\prime}$ for $i=1,2$ are analytic functions at $x=0$. Then for any $k\in \N$,
\begin{equation*}
\begin{aligned}
\int_{\Gamma_{h}^{+}}F(x)y^{k-2}dx=\int_{\Gamma_{h}^{+}}G_{1}(x) y^{k}dx,\\
\int_{\Gamma_{h}^{-}}F(x)y^{k-2}dx=\int_{\Gamma_{h}^{-}}G_{2}(x) y^{k}dx,
\end{aligned}
\end{equation*}
where 
$G_{i}(x)=\displaystyle\frac{2}{k}\left( \dfrac{\chi_{i}F}{\Psi^{\prime}_{i}}\right)^{\prime}(x)- 
\left( \dfrac{\chi^{\prime}_{i}F}{\Psi^{\prime}_{i}}\right)(x),$
for $i=1,2$.
\end{lemma}
\begin{proof}
If $(x,y)\in \Gamma_{h}^{+}$, then we have $\displaystyle\frac{dy}{dx}=-\dfrac{ \Psi_{1}^{\prime}(x)+\chi_{1}(x)y^2}{2\chi_{1}(x)y}$,
so
\begin{equation*}
\begin{aligned}
d(g_{1}(x)y^{k})=&\;g^{\prime}_{1}(x)y^{k}dx+k g_{1}(x)y^{k-1}dy\\
=&\left( g^{\prime}_{1}(x)-\frac{k}{2}\left( \dfrac{\Psi^{\prime}_{1}g_{1}}{\chi_1}\right) (x)\right) y^{k}dx
-\frac{k}{2} \left( \dfrac{\Psi^{\prime}_{1}g_{1}}{\chi_1}\right) (x)y^{k-2} dx,
\end{aligned}
\end{equation*}
and noting that 
\begin{equation*}
\int_{\overrightarrow{BA} }\left( g^{\prime}_{1}(x)-\frac{k}{2}\left( \dfrac{\Psi^{\prime}_{1}g_{1}}{\chi_1}\right) (x)\right) y^{k}dx=\int_{\overrightarrow{BA}}
\left( \dfrac{\Psi^{\prime}_{1}g_{1}}{\chi_1}\right) (x)y^{k-2} dx=0,
\end{equation*}
then we obtain that
\begin{equation*}
\begin{aligned}
\oint_{\Gamma_{h}^{+}\cup \overrightarrow{BA}}d(g_{1}(x)y^{k})=&\;\oint_{\Gamma_{h}^{+}\cup \overrightarrow{BA}}\left( g^{\prime}_{1}(x)-\frac{k}{2}
\left( \dfrac{\Psi^{\prime}_{1}g_{1}}{\chi_1}\right) (x)\right) y^{k}dx\\
&-\frac{k}{2}\oint_{\Gamma_{h}^{+}\cup \overrightarrow{BA}} \left( \dfrac{\Psi^{\prime}_{1}g_{1}}{\chi_1}\right) (x)y^{k-2} dx=0,
\end{aligned}
\end{equation*}
Now the result follows taking $F=\displaystyle\frac{k}{2}\left(\dfrac{\Psi^{\prime}_{1}g_{1}}{\chi_1}\right)(x) $ in the above equality. The proof of 
the second relation follows a similar line of reasoning to that of the first relation.
\end{proof}
Now we provide some new results to demonstrate the practical application of our work.
\begin{example} 
Consider the perturbed Hamiltonian differential system
\begin{equation}\label{syex1}
\left( \begin{array}{ll}
\dot x\\
\dot y \end{array}\right)= \begin{cases}
\left( \begin{array}{ll}
2 y\\
x(x-1)+\varepsilon (a^{+}_{0}+a^{+}_{1}x)y \end{array}\right), & x> 0,\\
\\
\left( \begin{array}{ll}
2 y\\
-2 x+\varepsilon (a^{-}_{0}+a^{-}_{1}x)y \end{array}\right), & x< 0,
\end{cases}
\end{equation}
with the Hamiltonian function
\begin{equation*}
H(x,y)=\begin{cases}
H^{+}(x,y)={y}^{2}+\Psi_{1}(x), & x> 0,\\
H^{-}(x,y)={y}^{2}+\Psi_{2}(x), & x< 0,
\end{cases}\;\ \text{with}\;\
\begin{cases}
\Psi_{1}(x)=\frac{1}{2}\,{x}^{2}-\frac{1}{3}\,{x}^{3},\\
\Psi_{2}(x)={x}^{2},
\end{cases}
\end{equation*}
where the orbital arcs $ \Gamma^{\pm}_{h} $ are defined for $h\in(0, 1/6)$; see Figure \ref{fig3}. The x-axis projection of the non-smooth period annulus 
satisfies $-{\frac{1}{\sqrt{6}}} < \sigma_1(x) < 0 < x < 1$. Furthermore, we observe that
\begin{figure}[]
\centering
\includegraphics[scale=.5]{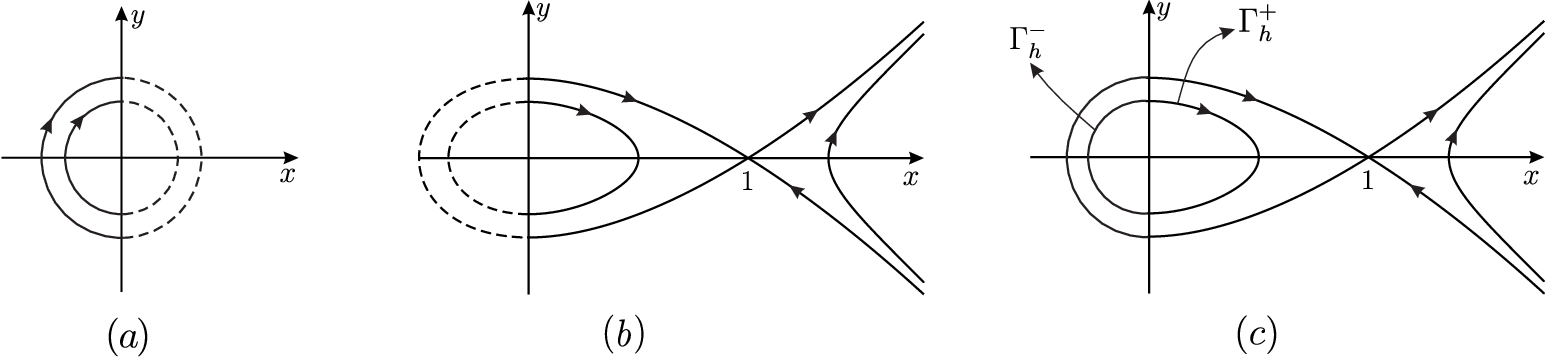}
\caption{$(a)$, $(b)$ and $(c)$ represent the level curves of $H^{-}$, $H^{+}$ and $H$, respectively.\label{fig3}}
\end{figure}
\begin{equation*}
\begin{array}{ll}
x\Psi_{1}^{\prime}(x)=x^{2}(1-x)>0,\hspace*{1cm}&\forall x\in(-\infty ,1)\setminus \lbrace 0\rbrace,\\
x\Psi_{2}^{\prime}(x)=2x^{2}>0,&\forall x\neq 0.
\end{array}
\end{equation*}
We can express the first order Melnikov function of system \eqref{syex1} as
\begin{equation*}
a^{+}_{0}I_{0}(h)+a^{-}_{0}J_{0}(h)+a^{+}_{1}I_{1}(h)+a^{-}_{1}J_{1}(h),
\end{equation*}
where
\begin{equation*}
I_{i}(h)=\int_{\Gamma_{h}^{+}}x^{i}ydx,\quad J_{i}(h)=\int_{\Gamma_{h}^{-}}x^{i}ydx,\qquad i=0,1.
\end{equation*}
First we show that the set $\lbrace I_0(h), J_0(h), I_1(h), J_1(h)\rbrace$ is linearly independent in the interval $(0, 1/6)$. We do 
this by showing that the only solution to the equation
\begin{equation}\label{lcex1}
c_{0}I_0(h)+c_{1}J_0(h)+c_{2}I_1(h)+c_{3}J_1(h)=0,
\end{equation}
for all $h\in(0, 1/6)$ is the trivial solution $c_{0}=c_{1}=c_{2}=c_{3}=0$.
Assuming that the non-smooth oval $\Gamma_{h}^{+}\cup \Gamma_{h}^{-}$ has a clockwise orientation, we can use the fact that 
$y_{h}^{\pm}(x)=-\tilde{y}_{h}^{\pm}(x)$ to obtain
\begin{equation}\label{II}
I_{i}(h)=2 \int_{0}^{x_{h}^{+}}x^{i}\tilde{y}_{h}^{+}(x) dx,\quad i=1,0,
\end{equation}
and 
\begin{equation*}
J_{i}(h)=2\int_{x_{h}^{+}}^{0}(\sigma_{1}(u))^{i}\sigma_{1}^{\prime}(u)y
\Big{\vert}_{\tilde{y}_{h}^{-}(\sigma_{1}(u))} du,
\end{equation*}
where in the above equality we applied the change of variable $x = \sigma_{1}(u)$. Now using the fact $\tilde{y}_{h}^{-}(\sigma_{1}
(u))=\tilde{y}_{h}^{+}(u)$, we have that
\begin{equation}\label{JJ}
J_{i}(h)=-2\int_{0}^{x_{h}^+}(\sigma_{1}(x))^{i}\sigma^{\prime}_{1}(x)\tilde{y}_{h}^{+}(x)dx,\quad i=1,0.
\end{equation}
Therefore, using \eqref{II}, \eqref{JJ} and equation \eqref{lcex1}, we can obtain that
\begin{equation*}
\int_{0}^{x_{h}^+}\left( c_{0} - c_{1}\sigma^{\prime}_{1}(x)+c_{2}x-c_{3}\sigma^{\prime}_{1}(x)\sigma_{1}(x)\right) 
\tilde{y}_{h}^{+}(x) dx=0,
\end{equation*}
and using the fact $ \sigma^{\prime}_{1}(x)=\dfrac{\Psi^{'}_{1}(x)}{\Psi^{'}_{2}(z)} $, we get
\begin{equation*}
p(x,y)=\left( c_{0}+c_{2}x\right)\Psi^{'}_{2}(z) - \left( c_{1}+c_{3}z\right) \Psi^{'}_{1}(x)=0,
\end{equation*}
where $z=\sigma_1(x)$ for $x\in(0,1)$ satisfies
\begin{equation*}
\Psi_{1}(x)-\Psi_{2}(z)=\frac{1}{6}q(x,z)=0,
\end{equation*}
with $q(x,z)=-2\,{x}^{3}+3\,{x}^{2}-6\,{z}^{2}$. The resultant of $ p(x,y)$ and $q(x,y)$ is given by
\begin{equation*}
\begin{aligned}
&\left( 12\,{c_{{0}}}^{2}-6\,{c_{{1}}}^{2} \right) {x}^{2}+ \left( -8
\,{c_{{0}}}^{2}+24\,c_{{0}}c_{{2}}-12\,c_{{0}}c_{{3}}+12\,{c_{{1}}}^{2
} \right) {x}^{3}\\
&+ \left( -16\,c_{{0}}c_{{2}}+20\,c_{{0}}c_{{3}}-6\,{c
_{{1}}}^{2}+12\,{c_{{2}}}^{2}-12\,c_{{2}}c_{{3}}+3\,{c_{{3}}}^{2}
\right) {x}^{4}\\
&+ \left( -8\,c_{{0}}c_{{3}}-8\,{c_{{2}}}^{2}+20\,c_{{2
}}c_{{3}}-8\,{c_{{3}}}^{2} \right) {x}^{5}+ \left( -8\,c_{{2}}c_{{3}}+
7\,{c_{{3}}}^{2} \right) {x}^{6}\\
&-2\,{c_{{3}}}^{2}{x}^{7}
=0.
\end{aligned}
\end{equation*}
Thus, for all $x$ in the interval $(0,1)$, we have $c_0=c_1=c_2=c_3=0$. Consequently, the family of functions $(I_0, J_0, I_1, J_1)$ 
is linearly independent in the interval $(0,1/6)$.

Since $s=1$ and $n=2$, it is clear that the hypothesis $s>2(n-1)$ is not satisfied. However, we can resolve this issue by using 
Lemma \ref{lems} to obtain new line integrals that satisfy the inequality with a sufficiently large corresponding value of $s$. 
Here we need to promote the power $s$ to three such that the hypothesis $s>2(n-1)$ holds.
On the arc $\Gamma_{h}^{+}$, we get that
\begin{equation*}
\begin{aligned}
I_{i}(h)=&\int_{\Gamma_{h}^{+}}x^{i}ydx=\frac{1}{h}\int_{\Gamma_{h}^{+}}\left( y^{2}+\Psi_{1}(x)\right)x^{i} ydx\\
=&\frac{1}{h}\left( \int_{\Gamma_{h}^{+}}x^{i}\Psi_{1}(x) ydx+\int_{\Gamma_{h}^{+}}x^{i}y^{3}dx\right),
\end{aligned}
\end{equation*}
and also, on the arc $\Gamma_{h}^{-}$, we obtain that
\begin{equation*}
\begin{aligned}
J_{i}(h)=&\int_{\Gamma_{h}^{-}}x^{i}ydx=\frac{1}{h}\int_{\Gamma_{h}^{-}}\left( y^{2}
+\Psi_{2}(x)\right)x^{i}ydx \\
=&\frac{1}{h}\left( \int_{\Gamma_{h}^{-}}x^{i}\Psi_{2}(x)ydx+\int_{\Gamma_{h}^{-}}x^{i}y^{3}dx\right).
\end{aligned}
\end{equation*}
We use Lemma \ref{lems} with $k=3$ and $F(x)=x^{i}\Psi_{j}(x)$ for $j=1,2$ to the first integral of $I_{i}(h)$ and $J_{i}(h)$,
respectively, to promote the value of $s$ as
\begin{equation*}
\int_{\Gamma^{+}_h}x^{i}\Psi_{1}(x)ydx=\int_{\Gamma^{+}_h}G_{1i}(x)y^3dx,\qquad \int_{\Gamma^{-}_h}x^{i}\Psi_{2}(x)ydx
=\int_{\Gamma^{-}_h}G_{2i}(x)y^3dx,
\end{equation*}
where
\begin{equation*}
\begin{aligned}
G_{1i}(x)=&\;\frac{2}{3}\left( \frac{x^{i}\Psi_{1}(x)}{\Psi^{\prime}_{1}(x)}\right)' =\frac{1}{9}\,{\frac {{x}^{i} 
\left( 2\,i{x}^{2}-5\,ix+2\,{x}^{2}+3\,i-4\,x+3 \right) }{ \left( x-1 \right) ^{2}}},\\
G_{2i}(x)=&\;\frac{2}{3}\left( \frac{x^{i}\Psi_{2}(x)}{\Psi^{\prime}_{2}(x)}\right)'=
\frac{1}{3}\,{x}^{i} \left( i+1 \right).
\end{aligned}
\end{equation*}
Now, the expressions for $I_{i}(h)$ and $J_{i}(h)$ take the following form
\begin{equation*}
\begin{aligned}
I_{i}(h)=&\;\frac{1}{h}\int_{\Gamma_{h}^{+}}\left( x^{i}+G_{1i}(x)\right) y^{3}dx
=\frac{1}{h^2}\int_{\Gamma_{h}^{+}}\left( y^{2}+\Psi_{1}(x)\right)\left( x^{i}+G_{1i}(x)\right) y^{3}dx\\
=&\;\frac{1}{h^2}\left( \int_{\Gamma_{h}^{+}}\left( x^{i}+G_{1i}(x)\right)\Psi_{1}(x) y^{3}dx+
\int_{\Gamma_{h}^{+}}\left( x^{i}+G_{1i}(x)\right) y^{5}dx\right), 
\end{aligned}
\end{equation*}
and 
\begin{equation*}
\begin{aligned}
J_{i}(h)=&\;\frac{1}{h}\int_{\Gamma_{h}^{-}}\left( x^{i}+G_{2i}(x)\right) y^{3}dx
=\frac{1}{h^2}\int_{\Gamma_{h}^{-}}\left( y^{2}+\Psi_{2}(x)\right)\left( x^{i}+G_{2i}(x)\right) y^{3}dx\\
=&\;\frac{1}{h^2}\left( \int_{\Gamma_{h}^{-}}\left( x^{i}+G_{2i}(x)\right)\Psi_{2}(x) y^{3}dx+
\int_{\Gamma_{h}^{-}}\left( x^{i}+G_{2i}(x)\right) y^{5}dx\right). 
\end{aligned}
\end{equation*}
We again apply Lemma \ref{lems} with $k=5$ and
$F(x)=(x^{i}+G_{ji}(x))\Psi_{j}(x)$, for $j=1,2$, to the first integral of $I_{i}(h)$ and $J_{i}(h)$,
respectively, to get that
\begin{equation*}
\begin{aligned}
\int_{\Gamma^{+}_h}(x^{i}+G_{1i}(x))\Psi_{1}(x)y^{3}dx=&\; \int_{\Gamma^{+}_h}H_{1i}(x)y^5dx,\\
\int_{\Gamma^{-}_h}(x^{i}+G_{2i}(x))\Psi_{2}(x)y^{3}dx=&\; \int_{\Gamma^{-}_h}H_{2i}(x)y^5dx,
\end{aligned}
\end{equation*}
where
\begin{equation*}
\begin{aligned}
H_{1i}(x)=&\;\frac{2}{5}\left( \frac{(x^{i}+G_{1i}(x))\Psi_{1}(x)}{\Psi^{\prime}_{1}(x)}\right)'
={\frac {h_{1i}(x) }{135 \left( x-1 \right) ^{4}}},\\
H_{2i}(x)=&\;\frac{2}{5}\left( \frac{(x^{i}+G_{2i}(x))\Psi_{2}(x)}{\Psi^{\prime}_{2}(x)}\right)'
={\frac {2 }{135}}\left({x}^{i} \left( 2\,{i}^{2}+13\,i+11 \right) \right),
\end{aligned}
\end{equation*}
with
\begin{equation*}
\begin{aligned}
h_{1i}(x)=&\;{x}^{i} ( 4\,{i}^{2}{x}^{4}-20\,{i}^{2}{x}^{3}+26\,i{x}^{4}+37\,{i}^{2}{x}^{2}-115\,i{x}^{3}+22\,{x}^{4}-30\,{i}^{2}x\\
&+194\,i{x}^{2}-88\,{x}^{3}+9\,{i}^{2}-150\,ix+141\,{x}^{2}+45\,i-108\,x+36).
\end{aligned}
\end{equation*}
Then the expressions for $I_i(h)$ and $J_i(h)$ can be written as follows
\begin{equation*} 
\begin{aligned}
I_i(h) &= \frac{1}{h^2} \int_{\Gamma_h^+} \left(x^i + G_{1i}(x) + H_{1i}(x)\right) y^5 dx, \\ 
J_i(h) &= \frac{1}{h^2} \int_{\Gamma_h^-} \left(x^i + G_{2i}(x) + H_{2i}(x)\right) y^5 dx.
\end{aligned}
\end{equation*}
Alternatively, we can define $L_{i}(h)$ as
\begin{equation*} 
L_i (h)=
\begin{cases}
\tilde{I}_i(h), & i = 0, 2, \\ 
\tilde{J}_i(h), & i = 1, 3,
\end{cases} 
\end{equation*}
where
\begin{equation*}
\tilde{I}_{i}(h) = \int_{\Gamma_h^+} f_{1i}(x) y^5 dx, \qquad
\tilde{J}_{i}(h) = \int_{\Gamma_h^-} f_{2i}(x) y^5 dx, 
\end{equation*}
and $f_{1i}(x)$ and $f_{2i}(x)$ are given by
\begin{align*} 
f_{1i}(x) &= x^i + G_{1i}(x) + H_{1i}(x), \\ 
f_{2i}(x) &= x^i + G_{2i}(x) + H_{2i}(x). 
\end{align*}
It is clear that $\lbrace I_0, J_0, I_1, J_1\rbrace$ is an ECT-system on $(0,1/6)$ if and only if
$\lbrace L_0,L_1,L_2,L_3\rbrace$ is as well. Now we can apply Theorem \ref{the1}, because 
$s=3$ and the condition $s>2(n-1)$ holds. Thus, by setting
\begin{equation*} 
l_i(x)= 
\begin{cases}
f_{1\left( \frac{i}{2}\right)}(x), & i=0,2,\\ 
\\
-f_{2\left( \frac{i-1}{2}\right) }(\sigma(x))\sigma^{\prime}(x), & i=1,3,
\end{cases} 
\end{equation*} 
we need to check that $\lbrace l_{0}, l_{1}, l_{2}, l_{3}\rbrace$ is a CT-system on $(0,1)$. In fact, we will show that 
$\lbrace l_{0}, l_{1}, l_{2}, l_{3}\rbrace$ is an ECT-system because a continuous Wronskian is easier to study. 
The Wronskian associated to $l_0$
is given by 
\begin{equation*}
W[l_0](x)=\frac{1}{135}{\frac {187\,{x}^{4}-748\,{x}^{3}+1146\,{x}^{2}-798\,x+216}{ \left( x-
1 \right) ^{4}}},
\end{equation*}
which is well-defined, and using Sturm's Theorem, it has no zeros in the interval $(0,1)$. It conclude that $W[l_0]\neq 0$ 
for all $x\in(0,1)$. The Wronskian associated to $l_0$ and $l_{1}$ is the rational function 
\begin{equation*}
W[\bold{l}_{\bold{2}}](x)=\frac{2}{675}{\frac {M_{{1}} \left( x,z \right)}{\left( x-1 \right) ^{4}{z}^{3}
}},
\end{equation*}
where
\begin{equation*}
\begin{aligned}
M_{{1}} \left( x,z \right) =&\;187\,{x}^{8}-1122\,{x}^{7}+748\,{x}^{5}{z}^{2}+2829\,{x}^{6}-3366\,{x}
^{4}{z}^{2}-3838\,{x}^{5}+6176\,{x}^{3}{z}^{2}\\
&+2958\,{x}^{4}-5688\,{x}^{2}{z}^{2}-1230\,{x}^{3}+2592\,x{z}^{2}+216\,{x}^{2}-432\,{z}^{2},
\end{aligned}
\end{equation*}
and it is clear that $ W[\bold{l}_{\bold{2}}](x)$ is well-defined in $-\sqrt [3]{\frac{1}{6}}<z<0<x <1$. 
The resultant with respect to $z$ between $q(x,z)$ and $ M_{{1}} \left( x,z \right) $ is $r_{1}(x)=16\,{x}^{6} p_{1}(x)$,
where
\begin{equation*}
p_1(x)=\left( 187\,{x}^{5}-1122\,{x}^{4}+2738\,{x}^{3}-3438\,{x}^{2}+2250\,x
-630 \right) ^{2}.
\end{equation*}
By applying Sturm's Theorem we get that $p_1(x)\neq 0$ for all $x\in (0,1)$. It implies that $W[\bold{l}_{\bold{2}}](x)\neq 0$ 
for all $x\in(0,1)$. Now we can find that
\begin{equation*}
W[\bold{l}_{\bold{3}}](x)=\frac{1}{91125}\frac { M_{2}(x,z)}{{z}^{5} \left( x-1 \right) ^{8}},
\end{equation*}
where $M_2(x,z)$ is a polynomial with long expression in $(x,z)$. The resultant with respect to $z$ between $q(x,z)$ and 
$ M_{2}(x,z) $ is $r_2(x)=16\,{x}^{10}p_2(x)$, where $ p_2(x) $ is a polynomial in $x$ of degree $18$. By applying Sturm's Theorem, 
we get that $r_2(x)\neq 0$ for all $x\in (0,1)$. It follows that $W[\bold{l}_{\bold{3}}](x)\neq 0$ for all $x\in(0,1)$. 
Finally, we get that
\begin{equation*}
W[\bold{l}_{\bold{4}}](x)=-\frac{7}{546750}\frac { M_{3}(x,z)}{{z}^{7} \left( x-1 \right) ^{8}},
\end{equation*}
where $M_3(x,z)$ is a polynomial with long expression in $(x,z)$. The resultant with respect to $z$ between $q(x,z)$ and 
$ M_{3}(x,z) $ is $r_3(x)=784\,{x}^{14}p_3(x)$, where $ p_3(x) $ is a polynomial in $x$ of degree $18$. Using Sturm's Theorem, 
we obtain that $r_3(x)\neq 0$ for all $x\in (0,1)$. It follows that $W[\bold{l}_{\bold{4}}](x)\neq 0$ for all $x\in(0,1)$. 
Thus this shows that $\lbrace l_0,l_1,l_2,l_3\rbrace$ is an ECT-system on $(0, 1)$. According to Theorem \ref{the1},
system \eqref{syex1} can have a maximum of three limit cycles that bifurcate from its period annulus.
\end{example}
\begin{example} 
Consider the perturbed Hamiltonian differential system
\begin{equation}\label{syex2}
\left( \begin{array}{ll}
\dot x\\
\dot y \end{array}\right)= \begin{cases}
\left( \begin{array}{ll}
2 y+\varepsilon\left( \sum_{i+j=0}^{1}a^{+}_{ij}x^{i}y^{j}\right) \\
2x-1+\varepsilon\left( \sum_{i+j=0}^{1}b^{+}_{ij}x^{i}y^{j}\right) \end{array}\right), &x> 0,\\
\\
\left( \begin{array}{ll}
2 y+\varepsilon\left( \sum_{i+j=0}^{1}a^{-}_{ij}x^{i}y^{j}\right)\\
1+\varepsilon\left( \sum_{i+j=0}^{1}b^{-}_{ij}x^{i}y^{j}\right)\end{array}\right), &x< 0,
\end{cases}
\end{equation}
with the Hamiltonian function
\begin{equation*}
H(x,y)=\begin{cases}
H^{+}(x,y)=\frac{1}{2}\,{y}^{2}+\Psi_{1}(x), & x> 0,\\
H^{-}(x,y)=\frac{1}{2}\,{y}^{2}+\Psi_{2}(x), & x< 0,
\end{cases}\;\ \text{with}\;\
\begin{cases}
\Psi_{1}(x)=-{x}^{2}+x,\\
\Psi_{2}(x)=-x,
\end{cases}
\end{equation*}
where the arcs $ \Gamma^{\pm}_{h} $ are defined for $h\in(0, 1/4)$. The non-smooth period annulus projects onto the $x$-axis satisfies $-\frac{1}{4}<\sigma_{1}(x)<0<x <\frac{1}{2}$. We also see that there exists a periodic annulus around the origin because 
\begin{equation*}
\begin{array}{ll}
x\Psi_{1}^{\prime}(x)=x(-2x+1)>0,\hspace*{1cm}&\forall x\in(0 ,\frac{1}{2}),\\
x\Psi_{2}^{\prime}(x)=-x>0,&\forall x< 0.
\end{array}
\end{equation*}
The first order Melnikov function of system \eqref{syex2} is given by
\begin{equation*}
\tilde{a}_{01}I_{0}(h)+\tilde{b}_{01}J_{0}(h),
\end{equation*}
where $ \tilde{a}_{01}=a^{+}_{10}+ b^{+}_{10}$, $ \tilde{b}_{01}=a^{-}_{10}+ b^{-}_{10}$ and 
\begin{equation*}
I_{0}(h)=\int_{\Gamma_{h}^{+}}ydx,\quad J_{0}(h)=\int_{\Gamma_{h}^{-}}ydx.
\end{equation*}
We first show that the set $\lbrace I_0(h), J_0(h)\rbrace$ is linearly independent in the interval $(0, 1/4)$. In fact,
the only solution to the equation
\begin{equation}\label{lcex2}
c_{0}I_0(h)+c_{1}J_0(h)=0,
\end{equation}
for all $h\in(0, 1/4)$ is the trivial solution $c_{0}=c_{1}=c_{2}=c_{3}=0$. Suppose that the non-smooth oval $\Gamma_{h}^{+}\cup \Gamma_{h}^{-}$ has a 
clockwise orientation. By using the fact that 
$y_{h}^{\pm}(x)=-\tilde{y}_{h}^{\pm}(x)$, we obtain
\begin{equation}\label{II2}
I_{0}(h)=2 \int_{0}^{x_{h}^{+}}\tilde{y}_{h}^{+}(x) dx,
\end{equation}
and also, using the change of variable $x = \sigma_{1}(u)$, we have
\begin{equation*}
J_{i}(h)=2\int_{x_{h}^{+}}^{0}\sigma_{1}^{\prime}(u)y
\Big{\vert}_{\tilde{y}_{h}^{-}(\sigma_{1}(u))} du.
\end{equation*}
Now, using $\tilde{y}_{h}^{-}(\sigma_{1}(u))=\tilde{y}_{h}^{+}(u)$, we get 
\begin{equation}\label{JJ2}
J_{i}(h)=-2\int_{0}^{x_{h}^+}\sigma^{\prime}_{1}(x)\tilde{y}_{h}^{+}(x)dx,\quad i=1,0.
\end{equation}
Hence, using \eqref{II2}, \eqref{JJ2} and equation \eqref{lcex2}, we can find that
\begin{equation*}
\int_{0}^{x_{h}^+}\left( c_{0} - c_{1}\sigma^{\prime}_{1}(x)\right) 
\tilde{y}_{h}^{+}(x) dx=0,
\end{equation*}
or equivalently,
\begin{equation*}
c_{0} - c_{1}\sigma^{\prime}_{1}(x)=0.
\end{equation*}
From $ \sigma^{\prime}_{1}(x)=2x-1 $, it follows that $c_{0}=c_{1}=0$. Consequently, the family of functions 
$(I_0(h), J_0(h))$ is linearly independent in the interval $(0,1/4)$.

Since $s=1$ and $n=1$, it shows that the hypothesis $s>2(n-1)$ holds. Now we define $L_{i}(h)$ as
\begin{equation*} 
L_i (h)=
\begin{cases}
I_i(h), & i = 0, \\ 
J_i(h), & i = 1,
\end{cases} 
\end{equation*}
and by setting
\begin{equation*} 
l_i(x)= 
\begin{cases}
1, & i=0,\\ 
-2x+1, & i=1,
\end{cases} 
\end{equation*} 
we need to check that $\lbrace l_{0}, l_{1}\rbrace$ is a ECT-system on $(0,1)$.
The Wronskian associated to $l_0$ is clearly nonzero, and we compute the Wronskian associated to $l_0$ and $l_1$ as
\begin{equation*}
W[\bold{l}_{\bold{2}}](x)=
\begin{vmatrix} 
1 & -2x+1 \\ 0 & -2
\end{vmatrix},
\end{equation*}
which is also nonzero.
Hence, by applying Theorem \ref{the1}, system \eqref{syex2} has at most one limit cycle that bifurcate from its period annulus.
\end{example}
\begin{example}
Consider the perturbed Hamiltonian differential system
\begin{equation}\label{syex3}
\left( \begin{array}{ll}
\dot x\\
\dot y \end{array}\right)= \begin{cases}
\left( \begin{array}{ll}
2 y\\
x(x-1)+\varepsilon (a^{+}_{0}+a^{+}_{1}x)y \end{array}\right), & x> 0,\\
\\
\left( \begin{array}{ll}
2 y\\
-x(x+1)+\varepsilon (a^{-}_{0}+a^{-}_{1}x)y \end{array}\right), & x< 0,
\end{cases}
\end{equation}
with the Hamiltonian function
\begin{equation*}
H(x,y)=\begin{cases}
H^{+}(x,y)={y}^{2}+\Psi_{1}(x), & x> 0,\\
H^{-}(x,y)={y}^{2}+\Psi_{2}(x), & x< 0,
\end{cases}\;\ \text{with}\;\
\begin{cases}
\Psi_{1}(x)=\frac{1}{2}\,{x}^{2}-\frac{1}{3}\,{x}^{3},\\
\Psi_{2}(x)=\frac{1}{2}\,{x}^{2}+\frac{1}{3}\,{x}^{3},
\end{cases}
\end{equation*}
where the arcs $ \Gamma^{\pm}_{h} $ are defined for $h\in(0, 1/6)$. The x-axis projection of the non-smooth period 
annulus satisfies $-1 < \sigma_1(x) < 0 < x < 1$, where $ \sigma_1(x)=-x $. Furthermore, we can see that
\begin{equation*}
\begin{array}{ll}
x\Psi_{1}^{\prime}(x)=x^{2}(1-x)>0,\hspace*{1cm}&\forall x\in(-\infty ,1)\setminus \lbrace 0\rbrace,\\
x\Psi_{2}^{\prime}(x)=x^{2}(x+1)>0,&\forall x\in(-1 ,\infty)\setminus \lbrace 0\rbrace.
\end{array}
\end{equation*}
The first order Melnikov function of system \eqref{syex3} can be expressed as a linear combination of four integrals, given by
\begin{equation}\label{lcex3}
a^{+}_{0}I_{0}(h)+a^{-}_{0}J_{0}(h)+a^{+}_{1}I_{1}(h)+a^{-}_{1}J_{1}(h),
\end{equation}
where
\begin{equation*}
I_{i}(h)=\int_{\Gamma_{h}^{+}}x^{i} y dx,\quad J_{i}(h)=\int_{\Gamma_{h}^{-}}x^{i} y dx,\qquad i=0,1.
\end{equation*}

First we show that the set of functions $\lbrace I_i(h),J_i(h)\rbrace$ for $i=0,1$ is linearly dependent on the open 
interval $(0,1/6)$. We must show that there exist constants $c_{i}$ and $c_{i+1}$, not all equal to zero, such that the 
linear combination
\begin{equation*}
c_i I_k(h) + c_{i+1} J_k(h)= 0,
\end{equation*} 
holds for all $h$ in $(0,1/6)$. Considering the non-smooth oval $\Gamma_{h}^{+}\cup \Gamma_{h}^{-}$ oriented in a clockwise direction, the previous equation can 
be transformed into the following form
\begin{equation*}
\int_{0}^{x_{h}^+}\left( c_{i}x^{i}-c_{i+1}\sigma^{\prime}_{1}(x)(\sigma_{1}(x))^{i}\right) \tilde{y}_{h}^{+}(x) dx=0,
\end{equation*}
so it follows that
\begin{equation*}
c_{i}x^{i}-c_{i+1}\sigma^{\prime}_{1}(x)(\sigma_{1}(x))^{i}=0.
\end{equation*}
Given that $\sigma_{1}(x) = -x$ and $\sigma_1^\prime(x) = -1$, we can deduce that $c_{i} = (-1)^{i+1}c_{i+1}$.
This implies that the two functions are linearly dependent on $(0,1/6)$. In other words, one function can be expressed as a 
scalar multiple of the other. Hence the linear combination \eqref{lcex3} is reduced to the following form
\begin{equation*}
a_{0}I_{0}(h)+a_{1}I_{1}(h).
\end{equation*}
Since we have $s=1$, it is clear that the hypothesis $s>n-2$ holds in this case. By setting
\begin{align*}
l_i(x)=x^{i},\qquad i=0,1,
\end{align*}
we will show that $\lbrace l_0,l_1\rbrace$ is a ECT-system on $x\in(0,1)$, and it implies that $\lbrace I_0, I_1\rbrace$ is an ECT-system 
as well.
The Wronskian of function $l_0(x)$ is obviously nonzero and we need to compute the Wronskian of the functions
$l_0(x)$ and $l_1(x)$ as follows
\begin{equation*}
W[1,x]=
\begin{vmatrix} 
1 & x \\ 0 & 1 
\end{vmatrix},
\end{equation*}
which is also nonzero. Now, using Theorem \ref{the2}, we conclude that system \eqref{syex3} has at most one limit cycle that 
bifurcate from the period annulus.
\end{example}

\section*{Acknowledgments}
The first author is supported by FAPESP grant, process number 2022/07822-5. The second author is partially support by FAPESP grant,
process number 2019/09385-9.

\end{document}